\documentclass[12pt]{article}

\usepackage{diagbox}
\usepackage{mathtools}
\usepackage{bbm}
\usepackage{latexsym}
\usepackage{epsfig}
\usepackage{amsmath,amsthm,amssymb,enumerate}

\usepackage[a-1b]{pdfx}
\usepackage{hyperref}

\usepackage{color}

\def\nn{\nonumber}
\def\a{\alpha}  \def\d{\delta} 
\def\e{\varepsilon}    
\def\G{\Gamma}  
     \def\l{\lambda}
   \def\p{\pi}
\def\r{\rho}   
 \def\om{\omega}

\newtheorem{theorem}{Theorem}

\newtheorem{claim}{Claim}

\newcommand{\rdup}[1]{{\left\lceil #1\right\rceil }}
\newcommand{\rdown}[1]{{\left\lfloor #1\right \rfloor}}
\def\bin{\text{Bin}}

\newcommand{\brac}[1]{\left(#1\right)}
\newcommand{\cbrac}[1]{\left\{#1\right\}}
\newcommand{\sqbrac}[1]{\left[#1\right]}

\newcommand{\bfrac}[2]{\left(\frac{#1}{#2}\right)}

\newcommand{\set}[1]{\left\{#1\right\}}

\def\E{\mathbb{E}}

\def\Pr{\mathbb{P}}

\newcommand{\ignore}[1]{}

\def\diam{\text{diam}}
\def\nn{\nonumber}

\usepackage{tikz}
\usetikzlibrary{decorations.pathmorphing}
\usetikzlibrary{positioning}
\usetikzlibrary{arrows,automata}
\usetikzlibrary{shapes.misc}
\usetikzlibrary{backgrounds}
\usetikzlibrary{arrows,shapes}

\begin{document}
\date{}
\title{The intersection of a random geometric graph with an Erd\H{o}s-R\'enyi graph}
\author{Patrick Bennett\thanks{Department of Mathematics, Western Michigan University, Kalamazoo MI 49008, Research supported in part by Simons Foundation Grant \#426894.}\and Alan Frieze\thanks{Department of Mathematical Sciences, Carnegie Mellon University, Pittsburgh PA 15213, Research supported in part by NSF grant DMS1952285.} \and Wesley Pegden\thanks{Department of Mathematical Sciences, Carnegie Mellon University, Pittsburgh PA 15213, Research supported in part by NSF grant DMS1700365. }}
\maketitle

\begin{abstract}
  We study the intersection of a random geometric graph with an Erd\H{o}s-R\'enyi graph. Specifically, we generate the random geometric graph $G(n, r)$ by choosing $n$ points uniformly at random from $D=[0, 1]^2$ and joining any two points whose Euclidean distance is at most $r$. We let $G(n, p)$ be the classical Erd\H{o}s-R\'enyi graph, i.e. it has $n$ vertices and every pair of vertices is adjacent with probability $p$ independently. In this note we study $G(n, r, p):=G(n, r) \cap G(n, p)$. One way to think of this graph is that we take $G(n, r)$ and then randomly delete edges with probability $1-p$ independently.  We consider the clique number, independence number, connectivity, Hamiltonicity, chromatic number, and diameter of this graph where both $p(n)\to 0$ and $r(n)\to 0$; the same model was studied by Kahle, Tian and Wang (2023) for $r(n)\to 0$ but $p$ fixed.
\end{abstract}

\section{Introduction}
In this note we study the intersection of a random geometric graph with an Erd\H{o}s-R\'enyi graph. Specifically, we generate the random geometric graph $G(n, r)$ by choosing $n$ points uniformly at random from $D=[0, 1]^2$ and joining any two points whose Euclidean distance is at most $r$. See Penrose \cite{Pen} for research monograph on this model. For convenience we work on the torus, so for example we consider $(0, 0), (1, 0), (0, 1)$ and $(1, 1)$ to the the same point. We let $G(n, p)$ be the classical Erd\H{o}s-R\'enyi graph, i.e. it has $n$ vertices and every pair of vertices is adjacent with probability $p$ independently. In this note we study the graph we will call $G(n, r, p):=G(n, r) \cap G(n, p)$. One way to think of this graph is that we take $G(n, r)$ and then randomly delete edges with probability $1-p$ independently. 

Kahle, Tian and Wang \cite{KTW23} studied the clique number of a random graph which is more general than our $G(n, r, p)$ (but for a more limited range of the parameter $p$). In \cite{KTW23} they study the {\em noisy random geometric graph}, which is the result of taking a random geometric graph, deleting each edge independently with some probability, and also adding random edges that were not present. Thus, our model is the {\em deletion-only case} of the model in \cite{KTW23}. However, in \cite{KTW23} they consider the case where only $r=r(n)$ can be a function of $n$, but $p$ is a fixed constant.  We allow $p=p(n)\to 0$ as well.  For convenience we define the parameter 
\[
q:=\p r^2p,
\]
which is the probability that that two given vertices are adjacent. 

We begin with a discussion of the clique number of $G=G(n,r,p)$.
\begin{theorem}\label{thm:clique}
For each fixed $\e>0$ there exists $C=C_\e$ such that we have the following. If $nr^2 \ge \log n$ and $C/(nr^2) \le 1-p \le \log^{-3}n$, then 
\begin{equation}\label{eqn:clique}
    \frac{(2-\e)\log\sqbrac{(1-p)nr^2}}{1-p} \le \omega(G) \le \frac{2\log\sqbrac{(1-p)nr^2}}{1-p}
\end{equation}
\end{theorem}
We next consider connectivity.
\begin{theorem}\label{thm:connectivity}
    Let $\e>0$. There exists some $r_0>0$ such that we have the following for all $r\le r_0$. If $q \le (1-\e) \log n / n$, then w.h.p.~$G(n, r, p)$ is disconnected and in particular it has isolated vertices. On the other hand if we have $q \ge (1+\e) \log n / n$ then w.h.p.~$G(n, r, p)$ is connected.
\end{theorem}
We next consider Hamiltonicity. Our result here is not as precise as that for connectivity.
\begin{theorem}\label{thm:hamiltonicity}
There exists some $r_0>0$ and $K>0$ such that we have the following for all $r\le r_0$, if $q \ge K\log n / n$ then w.h.p.~$G(n, r, p)$ is Hamiltonian.
\end{theorem}
After this we consider the independence number.
\begin{theorem}\label{thm:ind}
    For each fixed $\e>0$ and fixed $p$ with $0<p<1$ we have the following. If $nr^2 \ge n^\e$, then 
    \begin{equation}\label{eqn:ind}
     \alpha(G) = \Theta\brac{r^{-2}\log_{1/(1-p)}(nr^2)}
\end{equation}
\end{theorem}
We then consider the chromatic number.
\begin{theorem}\label{thm:color}
    For each fixed $\e>0$ and fixed $p$ with $0<p<1$ we have the following. If $nr^2 \ge n^\e$, then 
    \begin{equation}\label{eqn:ind}
     \chi(G) = \Theta\brac{ \frac{nr^2}{\log_{1/(1-p)}(nr^2)}}
\end{equation}
\end{theorem}
We finally consider the diameter.
\begin{theorem}\label{thm: diameter}
    If $r<1/2$ and $nr^2p \ge \log^2 n$ then w.h.p.
    \[
\diam(G) = \Theta\brac{ r^{-1} + \frac{\log n}{\log(nr^2p)}}.
\]
\end{theorem}
\section{Degree sequence}

 The degree of a vertex is distributed as $\bin(n-1,q)$. We will use the following standard tool to show concentration.

\begin{theorem}[Chernoff--Hoeffding bound] \label{thm:chernoff}
    Let $X$ be distributed as $Bin(m, s)$ and $0 < \d < 1$. Then 
    \[
    \Pr(|X-ms|>\d ms) \le 2 \exp(-\d^2 ms/3)
    \]
\end{theorem}

Suppose $q= \om( \log n /n)$. Letting $\d = [4\log n / nq]^{1/2}=o(1)$, we see that 
\[
\Pr\Big(\Big|deg(v)- (n-1)q \Big| > \d (n-1)q\Big) \le 2 \exp(-\d^2 (n-1)q/3) = o(1/n).
\]
Thus by the union bound, w.h.p.~every vertex in $G(n, r, p)$ has degree $(1+o(1))(n-1)q = (1+o(1))nq$.

\section{Cliques: proof of Theorem \ref{thm:clique}}\label{cliques}
\begin{proof}
  We first prove the upper bound. We partition $D$ into $1/r^2$ cells with side length $r$, and we say that a {\em block} consists of a 10 by 10 set of cells. Each block is anchored by a single cell in the top left hand corner. Thus there are $1/r^2$ blocks. We first reveal the location of the points. The probability that the $i$th point belongs to a given block is $100r^2$, and thus the expected number of points in each block is $100nr^2$, and the Chernoff bounds tell us that the probability that a given block has more than $200nr^2$ points is at most $e^{-100nr^2/3}, $  which is $o(r^2)$ for $r\geq \left(\frac{\log n}{n}\right)^{1/2}$, and in particular this bound on $r$ implies that, by a union bound over the $1/(100r^2)$ blocks, no block has more than $200nr^2$ points.
  
  Any clique must have all of its vertices lying in one of our blocks. The graph induced on the vertices in a block is distributed as $G(n', p)$ where $n' \le 200nr^2$. For our upper bound we can take $n'=200nr^2$ since adding more points to a block cannot decrease the clique size. The expected number of cliques of size $k:=\frac{2\log\sqbrac{(1-p)nr^2}}{1-p}$ in  $G(n', p)$ is then at most
\begin{align*}
    \binom {n'}k p^{\binom k2} \le \brac{\frac{en' p^\frac{k-1}{2}}{k}}^k
& \le \brac{\frac{O(1) nr^2 p^{\frac{2\log\sqbrac{(1-p)nr^2}}{2(1-p)}}}{k}}^k \\
&= \brac{\frac{O(1) nr^2 \sqbrac{(1-p)nr^2}^{\frac{2\log p}{2(1-p)}}}{k}}^k\\
    & = \brac{\frac{O(1) nr^2 \sqbrac{(1-p)nr^2}^{-1}}{\frac{\log\sqbrac{(1-p)nr^2}}{1-p}}}^k\text{ since $\log p/(1-p)=-(1+o(1))$}\\
    & = \brac{\frac{O(1)  }{\log\sqbrac{(1-p)nr^2}}}^k\\
    & =\bfrac{O(1)}{\log C}^{\Omega(\log^2 n)}
\end{align*}
which (for large $C$) is small enough to beat a union bound over the blocks. Thus w.h.p.~$G$ has no clique of size $k$. This completes the proof of the upper bound. 

Now we prove the lower bound. Here we use cells with side length $r/\sqrt{2}$. By the pigeonhole principle one of the cells contains at least $n'':=nr^2/2$ points, which are all within distance $r$ of each other. W.h.p.~(see, for example, \cite{FK}) since $C/(nr^2) \le 1-p \le \log^{-3}n$,  $G(n'', p)$ has a clique of size at least 
\[
\frac{(2-\e/2)\log[(1-p)n'']}{1-p} \ge \frac{(2-\e)\log\sqbrac{(1-p)nr^2}}{1-p}.
\]
{\bf Explanation:} to justify the first expression for the clique number, observe that for small $\r$ we have $\alpha(G(n, \r)) = (2 \pm \e) \log (n\r)/\r$ and so for $p=1-\r$ close to 1 we have $\omega(G(n, p)=(2 \pm \e) \log (n(1-p))/(1-p)$. This completes the proof of the upper bound. 
\end{proof}
\section{Connectivity: proof of Theorem \ref{thm:connectivity}}
We will use $r \le r_0$, for example, to conclude that a ball of radius $r$ has volume $\pi  r^2$ (since we have wrap around, this ball would intersect itself if $r$ were too big).

\begin{proof}
    
Suppose first that $q \le (1-\e)\log n / n$.  Let $X$ be the number of isolated vertices, so 
\[
\E[X] = n(1-q)^{n-1} =n \exp \Big(-(n-1)q + O(nq^2)\Big) \ge (1+o(1))n^\e.
\]
We will now bound the expected number of pairs of isolated vertices. Fix two points $u,v \in [0, 1]^2$. Say the volume of $B_r(u) \cap B_r(v)$ is $x$. The probability that a third vertex $w$ is nonadjacent to both $u$ and $v$ is 
\[
1-\Big(2\pi r^2-x\Big)p -x\Big(1-(1-p)^2\Big) = 1 - 2\pi r^2p - xp(1-p) \le 1-2q.
\]
Thus the expected number of pairs $u, v$ that are both isolated is at most
\[
n^2 (1-2q)^{n-2} \le n^2 \exp\Big(-2(n-2)q\Big) = (1+o(1))(\E[X])^2.
\]
Thus, by the second moment method, w.h.p.~$X>0$ and we have isolated vertices.

Now suppose that $q= (1+\e) \log n / n$. Note that this case suffices for all $q \ge (1+\e) \log n / n$ since the property of being  connected is monotone in $p$ (and also in $r$). We use cells of side length $\l := \eta r$ for some $\eta \ll \e$. 
We say a cell is {\em good} if it has at least $g:=\eta^2 n  \l^2$ points in it, and {\em bad} otherwise. We say that a cell is {\em great} if it is has say $x \ge g$ points (i.e.~it is good) and it has a connected component on at least $x-g/2$ vertices. We say that such a component is a {\em great component}
\begin{claim}\label{clm:greatneighbor}
    W.h.p.~every vertex has a neighbor in a great component.
\end{claim}

\begin{proof}

Consider a fixed vertex $v$ and the ball $B_r(v)$. Now since the diameter of a cell is $\sqrt{2} \eta r < 2 \eta r$, we have that for any point $v'$ with $||v-v'|| < (1-2 \eta)r $, the cell containing $v'$ lies entirely in $B_r(v)$. Thus, the number of cells contained entirely in  $B_r(v)$ is at least
\[
k:= \left \lceil \frac{Vol(B_{(1-2 \eta)r}(v))}{\l^2} \right \rceil= \left \lceil \pi (1-2 \eta)^2 \eta^{-2} \right \rceil.
\]
Note that 
\begin{equation}\label{eqn:kbounds}
  \pi  \eta^{-2}-  4 \eta^{-1} \  \le k \le \pi  \eta^{-2}+1.
\end{equation}
So we have some cells $C_1, \ldots, C_k$ each contained in $B_r(v)$. Suppose that each cell $C_i$ has $j_i$ points in it (not counting $v$ itself which is of course in one of them). Let $\rho(j_i)$ be the probability that the graph induced on $C_i$ either fails to have a component on at least $j_i-g/2$ vertices, or else there is such a component but $v$ is not adjacent to any of vertices in it. 
Then the probability that our claim fails for the vertex $v$ is at most 
\begin{equation}\label{eqn:1stmoment}
   \sum_{\substack{0 \le j \le n-1 \\ j_1, \ldots, j_k \ge 0 \\ j_1 + \ldots + j_k =j}} \binom{n-1}{j_1, \ldots, j_k, n-1-j} \l^{2j} \brac{1-k\l^2}^{n-1-j} \prod_{i:j_i \ge g} \r(j_i).
\end{equation}
{\bf Explanation:} we set $j=j_1+\ldots j_k$ and sum over all possibilities for these values. We must choose which points go where which explains the multinomial coefficient. Then for each of the $j$ points going into one of the cells $C_1, \ldots, C_k$, it has a probability of $\l^2$ of being in the correct cell. The rest of our $n-j$ points each has a probability of $1-k\l^2$ of not being in any cell $C_1, \ldots, C_k$. Finally, in order for our claim to fail at $v$, within each cell $C_i$ having $j_i \ge g$ points, the event whose probability is defined as $\r(j_i)$ must happen (and these events are independent since the $C_i$ induce disjoint Erd\H{o}s-R\'enyi graphs). 

We first handle the terms of \eqref{eqn:1stmoment} with small $j$. Using the fact that $m! \ge (m/e)^m$ for all $m \ge 0$  (if we agree that $0^0=1$ and $0 \log 0 = 0$), we have
\[
(j_1)! \ldots (j_k)! \ge \bfrac{j_1}{e}^{j_1}  \ldots \bfrac{j_k}{e}^{j_k} = \exp\cbrac{j_1 \log \bfrac{j_1}{e} + \ldots  j_k \log \bfrac{j_k}{e}} \ge \exp\cbrac{j \log \bfrac{j}{ke} } = \bfrac{j}{ke}^j,
\]
and so
\[
\binom{n}{j_1, \ldots, j_k, n-j} \le \frac{n^j}{(j_1)! \ldots (j_k)!} \le \bfrac{ekn}{j}^j.
\]

For each $j$ there are at most $k^j$ choices for the $j_1, \ldots j_k$ summing to $j$. Thus, the sum of all terms in \eqref{eqn:1stmoment} for $j \le \eta k\l^2 n$ is at most
\begin{align}
    \sum_{\substack{0 \le j \le \eta k\l^2 n }} k^j \bfrac{ekn}{j}^j \l^{2j} \brac{1-k\l^2}^{n-j} & \le \exp\{-k\l^2 n\} \sum_{\substack{0 \le j \le \eta k\l^2 n }} \bfrac{e k^2 \l^2 n}{j(1-k\l^2)}^j.\label{eqn:smalljbound}
\end{align}
Using \eqref{eqn:kbounds} and $\l=\eta r$ gives
\begin{equation}\label{eqn:kldbounds}
   (\pi   - 4 \eta )r^2 \le k \l^2 \le (\pi  + \eta^2) r^2, 
\end{equation}
so by choosing $r_0$ not too large (and $\eta$ small) we can guarantee say $k \l^2 < 1/2$. Now the ratio of consecutive terms in \eqref{eqn:smalljbound} is
\[
\bfrac{e k^2 \l^2 n}{(j+1)(1-k\l^2)}^{j+1} \bfrac{e k^2 \l^2 n}{j(1-k\l^2)}^{-j} = \frac{1}{j+1} \bfrac{j}{j+1}^j  \frac{e k^2 \l^2 n}{1-k\l^2} \ge 2,
\]
since $nr^2=\Omega(\log n)$. So the sum is on the order of its last term. Thus \eqref{eqn:smalljbound} is at most (explanation follows)
\begin{align*}
   \exp\{-k\l^2 n\} \cdot O\brac{\bfrac{ek}{\eta(1-k\l^2)}^{\eta k\l^2 n} }&= O\brac{ \exp\cbrac{-\brac{1-\eta \log \bfrac{2e(\pi  \eta^{-2}+1)}{\eta}}k\l^2 n} } \\
   & =O\brac{ \exp\cbrac{-\brac{1-\eta \log \bfrac{2e(\pi  \eta^{-2}+1)}{\eta}}(\pi   - 4 \eta )r^2 n} }\\
   & = O\brac{ \exp\cbrac{-\frac{(1+\e/2) \log n}{p}} }=o(1/n)
\end{align*}
On the first line above we used $k\l^2 <1/2$ and \eqref{eqn:kbounds}, and on the second line we used \eqref{eqn:kldbounds}. The last line follows since for  $\eta \ll \e$ the coefficient $\brac{1-\eta \log \bfrac{2e(\pi  \eta^{-2}+1)}{\eta}}(\pi   - 4 \eta )$ from the line above is at least say $(1-\e/10)\pi $, and $\pi r^2n = nq/p = (1+\e)\log n / p$.

Now we take care of the terms of \eqref{eqn:1stmoment} with larger $j$. Suppose $j > \eta k\l^2 n$. For any $j_i\ge g$ we have 
\begin{align}
    \r(j_i) &\le (1-p)^{j_i - g/2} + \sum_{g/2 \le m \le j_i/2} \binom{j_i}{m} (1-p)^{m(j_i-m)}\nn\\
    & \le (1-p)^{j_i - g/2} + \sum_{g/2 \le m \le j_i/2} \bfrac{ej_i(1-p)^{j_i}}{m(1-p)^m}^m \nn\\
    & \le  (1-p)^{j_i - g/2} + \sum_{g/2 \le m \le j_i/2} \bfrac{2ej_i(1-p)^{j_i/2}}{g}^m\nn\\
&\leq (1-p)^{j_i - g/2} + \sum_{g/2 \le m \le j_i/2} \bfrac{4}{gp}^m = (1+o(1))(1-p)^{j_i - g/2},\label{eqn:rho}
\end{align}
 where the last line follows from $xe^{-ax}\leq 1/(ea)$ and $gp=\Omega(\log n)$.

The total sum of all values $j_i$ with $j_i < g$ is at most $kg$, and so the sum of the rest of the $j_i$ is at least $j-kg$. Thus the product in \eqref{eqn:1stmoment} is at most 
\[
\prod_{i:j_i \ge g} \r(j_i) \le \prod_{i:j_i \ge g} (1+o(1)) (1-p)^{j_i-g/2} \le (1+o(1)) (1-p)^{j-3kg/2}
\]
The sum of the terms of \eqref{eqn:1stmoment} with $j > \eta k\l^2 n$ is therefore bounded above by
\begin{align}
    &(1+o(1))\sum_{\substack{\eta k\l^2 n < j \le n \\ j_1, \ldots, j_k \ge 0 \\ j_1 + \ldots + j_k =j}} \binom{n}{j_1, \ldots, j_k, n-j} \l^{2j} \brac{1-k\l^2}^{n-j}  (1-p)^{j-3kg/2}\nonumber\\
    & \le (1+o(1)) \exp \{3kgp/2\} \sum_{\substack{0 \le j \le n \\ j_1, \ldots, j_k \ge 0 \\ j_1 + \ldots + j_k =j}} \binom{n}{j_1, \ldots, j_k, n-j} \l^{2j} \brac{1-k\l^2}^{n-j}  (1-p)^{j}\nonumber\\
    & = (1+o(1)) \exp \{3kgp/2\} \Big(k\l^2 (1-p) + (1-k\l^2)\Big)^n \label{eqn:bigjbound}
\end{align}
where the last line follows from the multinomial theorem. Simplifying and using $1+x \le e^x$, the above is at most 
\begin{align*}
    (1+o(1)) \exp \{3kgp/2- k\l^2 pn\} & \le (1+o(1)) \exp \cbrac{- \brac{1-\frac32  \eta^2  } k\l^2 pn}\\
    & \le (1+o(1)) \exp \cbrac{- \brac{1-\frac32  \eta^2  } (\pi   - 4 \eta )r^2 pn}\\
    & \le \exp\{-(1+\e/2) \log n\} = o(1/n)
\end{align*}
where on the first line we replaced $g=\eta^2 n\l^2$, on the second line we used \eqref{eqn:kldbounds}, and on the last line we used $\eta \ll \e$ and the fact that $q=\pi r^2p = (1+\e) \log n / n$. Thus \eqref{eqn:1stmoment}, our bound on the probability that the claim fails for a single vertex, is $o(1/n)$. The claim now follows from the union bound over $n$ vertices. 
\end{proof}

We now define an auxiliary graph $\G$ whose vertex set is our set of great cells in $G(n, r, p)$, and where two cells are adjacent in $\G$ if the diameter of their union is at most $r$ (i.e.~every point in one cell is distance at most $r$ from every point in the other cell). Note that if we have two cells adjacent in $\G$ the probability there is no edge between their great components is at most $(1-p)^{g^2/4}$, small enough to union bound over pairs of cells and conclude that w.h.p.~every such pair of cells has an edge between the great components. 

Recalling that we have Claim \ref{clm:greatneighbor}, to finish the proof of Theorem \ref{thm:connectivity} it suffices to prove the following:
\begin{claim}
    W.h.p.~$\G$ is connected.
\end{claim}

\begin{proof}
    The probability that any given cell is bad is at most 
\begin{align}
\sum_{j=0}^g  \binom{n}{j} \l^{2j} \brac{1-\l^2}^{n-j} & \le 2g \binom{n}{g} \l^{2g} \brac{1-\l^2}^{n-g} \nn\\
&\le 2g \exp \cbrac{g \log \bfrac{ne}{g} + g \log\brac{\l^{2}}- (n-g) \l^2}\nn\\
& \le 2g \exp \cbrac{\eta^2 n  \l^2 \brac{\log \bfrac{e}{\eta^2}+\l^2} - n \l^2}\nn\\
& \le \exp\cbrac{-(1-\e/10)n\l^2}.\label{eqn:badcellprob}
\end{align}
Note that if we condition on one cell being bad, the event that another cell is bad only becomes less likely. 

    We consider blocks composed of $100 \eta^{-4}$ cells as defined in Section \ref{cliques}. The probability that a block contains $k-10\eta^{-1}$ bad cells is at most 
    \begin{align*}
        \binom{100 \eta^{-4}}{k-10\eta^{-1}} \exp\cbrac{-(k-10\eta^{-1})(1-\e/10)n\l^2} &= O\brac{\exp\cbrac{-(\p \eta^{-2}-14\eta^{-1})(1-\e/10)n(\eta r)^2}}\\
        &= O\brac{\exp\cbrac{-(1-\e/10)(\p-14\eta)nr^2}}\\
        & = O\brac{\exp\cbrac{-(1+\e/2)\log n / p}}=o(1/n)
    \end{align*}
    where we have used \eqref{eqn:kbounds} and $\eta \ll \e$. Since the number of blocks is $O( (\eta^2 r)^{-2} ) = o(n)$, by the union bound we have that no block contains $k-10\eta^{-1}$ bad cells. 

Following part of the same calculation in \eqref{eqn:rho}, the probability that a good cell having $j \ge g$ points fails to be great is at most 
\begin{align*}
    \sum_{g/2 \le m \le j/2} \binom{j}{m} (1-p)^{m(j-m)} &\le \sum_{g/2 \le m \le j/2} \bfrac{2ej(1-p)^{j/2}}{g}^m\\
    & \le \frac j2 \bfrac{4}{gp}^m\leq e^{-\Omega(\log n\log\log n)},
\end{align*}
and so by the union bound w.h.p.~every good cell is great.

    Now within a block there are $10\eta^{-2} > 2k$ rows and columns (i.e. vertical or horizontal strips of cells). Thus, within a block we have that more than half of the rows (and more than half of the columns) contain only great cells. Call such a row or column great. Thus, for any two blocks that are disjoint except for sharing a side, there is some great row or column in one of the blocks touching a great row or column in the other block. Thus, the cells that are in great rows or columns are all in the same component of $\G$. Any other component of $\G$ must be contained entirely in one block and consist of cells bounded by great rows and columns. Suppose there is such a component $\G_0$ of $\G$. Consider a ``highest'' cell $C$ in $\G_0$, i.e.~a cell whose $y$-coordinate of its center $(x_C, y_C)$ is largest. There are at least $k$ cells which contain only points of distance at most $r$ from $(x_C, y_C)$. At most $2\eta^{-1}$ of these cells are at the same height as $C$, and so at least $(k-2\eta^{-1})/2$ of them are above $C$. These $(k-2\eta^{-1})/2$ cells must not be great. We can find an additional  $(k-2\eta^{-1})/2$ cells which are not great by considering the ``lowest'' cell in $\G_0$. Thus, $\G_0$ is contained in a block which has $k-2\eta^{-1}$ cells which are not great, which is a contradiction. This completes the proof of our claim, and the proof of Theorem \ref{thm:connectivity}.

\end{proof}

\end{proof}
\section{Hamilton Cycles: proof of Theorem \ref{thm:hamiltonicity}}
\begin{proof}
We begin by decomposing $G_{n,p}$ into $G_{n,p_1}\cup G_{n,p_2}$ where $p_1=p_2$ and $1-p=(1-p_1)(1-p_2)$. Then we partition $D$ into cells $C_1,C_2,\ldots,C_m,m=5/r^2$ of side $r/\sqrt{5}$. We assume that the cell order is such that it defines a sequence where $C_i,C_{i+1}$ share an edge for $1\leq i<m$. Now the expected number of points in a cell is $nr^2/5\geq K\log n/(5p)$. So we can assume that $K$ is large enough so that w.h.p. every cell has at least $10\log n/p$ points.

Now if the are $n_i$ points in a cell $i$ then these points induce a random graph $\G_i=G_{n_i,p_1}$ where we have assumed that $n_i\geq 10\log n/p$ for $i\in[m]$. It is not difficult to show that $\Pr(G_{N,\xi}\text{ is not Hamiltonian})\leq 2N^2qe^{-Nq}$ when $\xi=O(\log N/N)$ (see for example Chapter 6 of Frieze and Karo\'nski \cite{FK} where we see that this probability is dominated by the probability of the existence of vertices of degree less than two). It follows from this that w.h.p., $\G_i$ is Hamilton for $i\in[m]$. Indeed,
\[
\Pr(\exists i:\G_i\text{ is not Hamiltonian})\leq 2\sum_{i=1}^mn_i^2p_1e^{-n_ip_1}\leq 2n^3p_1n^{-10p_1/p}=o(1),
\]
since $p_1,p_2\geq p/2$.

Assume then that $\G_i$ contains a Hamilton cycle for $i\in[m]$. Divide $C_i$ into two paths $P_{i,-},P_{i,+}$ of length $\rdown{n/2}$ or $\rdup{n/2}$. We will show next that w.h.p. for all $i\in[m-1]$ there is are edges $\set{a_i,b_i}\in P_{i,+},\set{c_i,d_i}\in P_{i+1,-}$ such that $G_{n,p_2}$ contains both edges $\set{a_i,c_i},\set{b_i,d_i}$. If for $i\in[m-1]$, we delete $\set{a_i,b_i},\set{c_i,d_i}$ from $\bigcup_{i=1}^mC_i$ and replace them by $\set{a_i,c_i},\set{b_i,d_i}$ then we obtain a Hamilton cycle. The probability that we fail to make this transformation can be bounded by
\[
\sum_{i=1}^{m-1}(1-p_2^2)^{n_in_j/16}\leq me^{-100p_2^2\log^2n/16p^2}=o(1).
\]
Here we have used every other edge in the paths $P_{i,+},P_{i,-}$ to avoid dependencies.
\end{proof}
\section{Independence number: proof of Theorem \ref{thm:ind} }

\begin{proof}
We start with the lower bound. We use cells of side length $r$. By removing every other row and every other column of cells, we obtain a set of $1/(4r^2)$ cells so that no two cells in our set intersect share a boundary (not even a corner). The distance between any two distinct cells is at least $r$, so no point in one would ever be adjacent to any point in another. By Chernoff, w.h.p.~ every cell in $C$ has at least $n'=nr^2/2$ points. W.h.p.~(see, for example, \cite{FK}) $G(n', p)$ has an independent set of size at least $\log_{1/(1-p)}(n')$. Thus, for constant $p$ w.h.p.~we have
\[
\a(G) \ge |C| \log_{1/(1-p)}(n') = \frac{1}{4r^2}\log_{1/(1-p)}(nr^2/2)= \Omega\brac{r^{-2}\log_{1/(1-p)}(nr^2)}.
\]
We move on to the upper bound. Use cells of side length $r/\sqrt{2}$, so there are $2r^{-2}$ cells. The Chernoff bounds imply that w.h.p.~ no cell has more than $n'':=2n\left(r/\sqrt{2} \right)^2=nr^{2}$ points. The expected number of independent sets of size $k:=3 \log _{1/(1-p)} n''$ in  $G(n'', p)$ is at most
\begin{align*}
    \binom {n''}k (1-p)^{\binom k2} \le \brac{n'' (1-p)^\frac{k-1}{2}}^k \le \brac{O(1) nr^2 (1-p)^{\frac 32 \log_{1/(1-p)}(nr^2)}}^k &= \brac{O( (nr^2)^{-1/2})}^k,
\end{align*} 
small enough for the union bound to conclude that w.h.p.~no cell has an independent set larger than $k$. 
Summing over the cells we have 
\[
\a(G) \le n'' k = 2r^{-2} \cdot 3 \log_{1/(1-p)}\brac{nr^{2}} = O\brac{r^{-2}\log_{1/(1-p)}(nr^2)}.
\]
This completes the proof of the upper bound. 
\end{proof}
\section{Chromatic number: proof of Theorem \ref{thm:color}}

\begin{proof}
First we prove the upper bound. Use cells of side length $r$. We will use $4$ disjoint palettes of colors. Each palette will be used on a set of $(2r)^{-2}$ cells in every other row and every other column, so that two cells using the same palette are always distance at least $r$ apart. The Chernoff bound implies that  w.h.p.~every cell has at most $n':=2nr^2$ points. We will use the following result:
\begin{theorem}[Theorems 7.7, 7.8 in \cite{FK}]
    For constant $0 < p < 1$ we have
    \[
    \E(\chi(G(n, p))) = (1+o(1)) \frac{n}{2 \log_{1/(1-p)} n}
    \]
    and 
    \[
    \Pr( |\chi(G(n, p) - \E(\chi(G(n, p)))| \ge t ) \le \exp(-t^2/2n).
    \]
\end{theorem}
Thus we have 
\[
\Pr\brac{ G(n', p) \ge \frac{n'}{ \log_{1/(1-p)} n'}} \le \exp\brac{-\frac{(1+o(1)) \brac{\frac{n'}{2 \log_{1/(1-p)} n'}}^2 }{2n'}} = \exp\brac{-\Omega\brac{\frac{n'}{\log^2 n'}}},
\]
which is small enough to beat the union bound over all cells. Thus, w.h.p. within each cell we have a graph which requires at most $n'/ \log_{1/(1-p)} n'$ colors. Thus each palette can have that many colors, and 
we have
\[
\chi(G) \le \frac{4 n'}{\log_{1/(1-p)}(n')} =  O\brac{ \frac{nr^2}{\log_{1/(1-p)}(nr^2)}}.
\]

Now we prove the lower bound. Using Theorem \ref{thm:ind}, we have

\[
\chi(G) \ge \frac{n}{\a(G)}  =\Omega\brac{ \frac{nr^2}{\log_{1/(1-p)}(nr^2)}}.
\]

\end{proof}

\section{Diameter: proof of Theorem \ref{thm: diameter}}
\begin{proof}
We first prove the lower bound. We use cells of side length $r$, and the Chernoff bounds imply that w.h.p.~every cell has a point in it. Taking two points in farthest-apart cells, they have Euclidean distance $\Omega(1)$ and so their graph distance is $\Omega(r^{-1})$. Thus 
\begin{equation}\label{eqn:diamLB1}
    \diam(G) = \Omega(r^{-1}).
\end{equation}
Now note that by the Chernoff bounds, w.h.p.~every vertex has degree at most $2nq = 2\pi nr^2p$, i.e.~about twice the expected degree. Assuming this, the number of vertices within a distance $i$ of a point $v$ is at most 
\[
1+2\pi nr^2p + \ldots + (2\pi nr^2p)^i = (1+o(1))(2\pi nr^2p)^i.
\]
Therefore if $(2\pi nr^2p)^i < n/2$ then $\diam(G) >i$. In particular, 
\begin{equation}\label{eqn:diamLB2}
    \diam(G) > \frac{\log(n/2)}{ \log(2\pi nr^2p)} = \Omega \brac{\frac{\log n}{\log(nr^2p)}}.
\end{equation}
Combining \eqref{eqn:diamLB1} and \eqref{eqn:diamLB2} completes our lower bound:
\[
\diam(G) = \Omega\brac{ r^{-1} + \frac{\log n}{\log(nr^2p)}}.
\]

Now we prove the upper bound. Here we will use cells of side length $r/2\sqrt{2}$. If two cells intersect at all (even at a single point) then every point in one cell is within distance $r$ of every point in the other cell. The Chernoff bounds imply that w.h.p.~every cell has at least $n':=n r^2/16$ points (half the expected number). 

Consider any two points $u, v$. Let
\[
\ell:= \max\cbrac{ \frac{\sqrt 2}{r} \;\;, \frac{\log(n'/10)}{\log (n'p/2)}}
\]
Since $\ell \ge \sqrt{2}/r$, there is a sequence of (not necessarily distinct) cells $C_0, \ldots, C_\ell$ such that $u \in C_0$, $v \in C_\ell$, and any pair of consecutive cells in our sequence intersects. Let $S_0=\{u\}$ and for $i=1, \ldots \ell$, let $S_i$ be some set of points in $C_i$ that have neighbors in $S_{i-1}$. We will show that w.h.p. we can take
\begin{equation}\label{eqn:Sibound}
    |S_i| = \min\cbrac{ \bfrac{n'p}{2}^i, \frac{n'}{10}}, \qquad i=0, \ldots, \ell.
\end{equation}
We will prove this by induction on $i$. The base case $i=0$ holds since $|S_0|=1$.  We have already revealed the location of all the points in the plane, so our inductive proof can proceed by revealing adjacencies between points in $S_{i}$ and points in $C_{i+1}$ (so each edge is present with probability $p$). For the induction step, assume it holds for some $i < \ell$. The number $X$ of vertices in $C_{i+1}$ and not in $S_i$ that are adjacent to some vertex in $S_i$ is distributed as $\bin\brac{n'', 1-(1-p)^{|S_i|}}$, where $n'' \ge 9n'/10$ is the number of points in $C_{i+1}$ and not $S_i$.

{\bf Case 1:} Suppose  we have  $\bfrac{n'p}{2}^{i+1}\le \frac{n'}{10}$. Then $|S_i| = \bfrac{n'p}{2}^{i} \le \frac{1}{5p},$ so 
\begin{align}
    \E[X] = n'' \sqbrac{ 1-(1-p)^{|S_i|}} &\ge \frac{9n'}{10}  \sqbrac{ 1-\exp\cbrac{-|S_i|p}}\nn\\
    &\ge \frac{9n'}{10}  \sqbrac{ |S_i|p - \frac 12 (|S_i|p)^2} \nn\\
    &\ge \frac 9{10} \sqbrac{ 1 - \frac 1{50}}|S_i|n'p > \frac 34 \bfrac{n'p}{2}^{i} n'p.\label{eqn:Xbound}
\end{align}
Since the above expectation is $\Omega(n'p) = \Omega(\log^2 n)$, the Chernoff bounds imply that with probability at least $1-\exp(-\Omega(\log^2 n ))$, the random variable $X$ is at least $\bfrac{n'p}{2}^{i+1}$ (2/3 of \eqref{eqn:Xbound}). The failure probability is small enough to beat the union bound over $\ell$ steps. This completes the proof for Case 1. 

{\bf Case 2:} Suppose  we have  $\bfrac{n'p}{2}^{i+1}> \frac{n'}{10}$. Then $|S_i| \ge \bfrac{n'p}{2}^{i} > \frac{1}{5p} ,$ so 
\begin{align}
    \E[X] = n'' \sqbrac{ 1-(1-p)^{|S_i|}} &\ge \frac{9n'}{10}  \sqbrac{ 1-\exp\cbrac{-|S_i|p}} > \frac{9n'}{10}  \sqbrac{ 1-\exp\cbrac{-1/5}} > 0.15 n'.\label{eqn:Xbound2}
\end{align}
Since the above is $\Omega(\log^2 n)$, again Chernoff gives us that with probability at least $1-\exp(-\Omega(\log^2 n ))$, the random variable $X$ is at least $\frac{n'}{10}$ (2/3 of \eqref{eqn:Xbound2}). This completes the proof for Case 2. Thus w.h.p.~\eqref{eqn:Sibound} holds.

Because $\ell \ge \frac{\log(n'/10)}{\log (n'p/2)}$, we have 
\[
\bfrac{n'p}{2}^\ell \ge \frac{n'}{10}
\]
and so $|S_\ell|= n'/10$. The probability there is no edge from $v$ to $S_{\ell}$ is then $(1-p)^{n'/10}= \exp\{-\Omega(n'p)\}= \exp\{-\Omega(\log^2 n)\}$, so w.h.p.~there is such an edge. Thus the distance from $u$ to $v$ is at most 
\[
\ell+1 = 1+ \max\cbrac{ \frac{\sqrt 2}{r} \;\;, \frac{\log(n'/10)}{\log (n'p/2)}} = O\brac{ r^{-1} + \frac{\log n}{\log(nr^2p)}}
\]
\end{proof}

\end{document}